\documentclass[11pt]{amsart}

\usepackage{latexsym}

\usepackage{amsmath,amsfonts,amssymb,amsthm,dsfont,bbm,mathrsfs}
\usepackage[english]{babel}
\usepackage[hidelinks]{hyperref}

\usepackage{etoolbox} % Use '\AfterEndEnvironment' to avoid indent after theorems

\usepackage[authoryear]{natbib}

\makeatletter

\def\cite{\citet}

\def\@noindentfalse{\global\let\if@noindent\iffalse}
\def\@noindenttrue {\global\let\if@noindent\iftrue}
\def\@aftertheorem{%
  \@noindenttrue
  \everypar{%
    \if@noindent%
      \@noindentfalse\clubpenalty\@M\setbox\z@\lastbox%
    \else%
      \clubpenalty \@clubpenalty\everypar{}%
    \fi}}

\theoremstyle{plain}
\newtheorem{prop}{Proposition}[section]
\AfterEndEnvironment{prop}{\@aftertheorem}

\AfterEndEnvironment{definition}{\@aftertheorem}
\newtheorem{lemma}[prop]{Lemma}
\AfterEndEnvironment{lemma}{\@aftertheorem}
\newtheorem{cor}[prop]{Corollary}
\AfterEndEnvironment{cor}{\@aftertheorem}
\newtheorem{theorem}[prop]{Theorem}
\AfterEndEnvironment{theorem}{\@aftertheorem}

\theoremstyle{definition}
\newtheorem{remark}[prop]{Remark}

\numberwithin{equation}{section}
\allowdisplaybreaks

\def\be#1{\begin{equation*}#1\end{equation*}}
\def\ben#1{\begin{equation}#1\end{equation}}
\def\bes#1{\begin{equation*}\begin{split}#1\end{split}\end{equation*}}
\def\besn#1{\begin{equation}\begin{split}#1\end{split}\end{equation}}

\def\ba#1{\begin{align*}#1\end{align*}}

\def\bg#1{\begin{gather*}#1\end{gather*}}

\def\given{\mskip 0.5mu plus 0.25mu\vert\mskip 0.5mu plus 0.15mu}
\newcounter{bracketlevel}%
\def\@bracketfactory#1#2#3#4#5#6{%
\expandafter\def\csname#1\endcsname##1{%
\global\advance\c@bracketlevel 1\relax%
\global\expandafter\let\csname @middummy\alph{bracketlevel}\endcsname\given%
\global\def\given{\mskip#5\csname#4\endcsname\vert\mskip#6}\csname#4l\endcsname#2##1\csname#4r\endcsname#3%
\global\expandafter\let\expandafter\given\csname @middummy\alph{bracketlevel}\endcsname%
\global\advance\c@bracketlevel -1\relax%
}%
}
\def\bracketfactory#1#2#3{%
\@bracketfactory{#1}{#2}{#3}{relax}{0.5mu plus 0.25mu}{0.5mu plus 0.15mu}
\@bracketfactory{b#1}{#2}{#3}{big}{1mu plus 0.25mu minus 0.25mu}{0.6mu plus 0.15mu minus 0.15mu}
\@bracketfactory{bb#1}{#2}{#3}{Big}{2.4mu plus 0.8mu minus 0.8mu}{1.8mu plus 0.6mu minus 0.6mu}
\@bracketfactory{bbb#1}{#2}{#3}{bigg}{3.2mu plus 1mu minus 1mu}{2.4mu plus 0.75mu minus 0.75mu}
\@bracketfactory{bbbb#1}{#2}{#3}{Bigg}{4mu plus 1mu minus 1mu}{3mu plus 0.75mu minus 0.75mu}
}
\bracketfactory{clc}{\lbrace}{\rbrace}
\bracketfactory{clr}{(}{)}
\bracketfactory{cls}{[}{]}
\bracketfactory{abs}{\lvert}{\rvert}
\bracketfactory{norm}{\Vert}{\Vert}
\bracketfactory{floor}{\lfloor}{\rfloor}
\bracketfactory{ceil}{\lceil}{\rceil}
\bracketfactory{angle}{\langle}{\rangle}

\newcounter{ctr}\loop\stepcounter{ctr}\edef\X{\@Alph\c@ctr}%
	\expandafter\edef\csname s\X\endcsname{\noexpand\mathscr{\X}}
	\expandafter\edef\csname c\X\endcsname{\noexpand\mathcal{\X}}
	\expandafter\edef\csname b\X\endcsname{\noexpand\boldsymbol{\X}}
	\expandafter\edef\csname I\X\endcsname{\noexpand\mathbbm{\X}}
\ifnum\thectr<26\repeat

\let\@IE\IE\let\IE\undefined
\newcommand{\IE}{\mathop{{}\@IE}\mathopen{}}
\let\@IP\IP\let\IP\undefined
\newcommand{\IP}{\mathop{{}\@IP}}
\newcommand{\Var}{\mathop{\mathrm{Var}}}
\newcommand{\TP}{\mathop{\mathrm{TP}}}
\newcommand{\Hyp}{\mathop{\mathrm{Hyp}}}
\newcommand{\Bi}{\mathop{\mathrm{Bi}}}
\newcommand{\Be}{\mathop{\mathrm{Be}}}
\newcommand{\cBi}{\mathop{\hat{\mathrm{Bi}}}}
\newcommand{\Cov}{\mathop{\mathrm{Cov}}}
\newcommand{\bigo}{\mathop{{}\mathrm{O}}\mathopen{}}

\newcommand{\law}{\mathop{{}\sL}\mathopen{}}

\let\original@left\left
\let\original@right\right
\renewcommand{\left}{\mathopen{}\mathclose\bgroup\original@left}
\renewcommand{\right}{\aftergroup\egroup\original@right}

\def\^#1{\relax\ifmmode {\mathaccent"705E #1} \else {\accent94 #1} \fi}
\def\~#1{\relax\ifmmode {\mathaccent"707E #1} \else {\accent"7E #1} \fi}
\def\*#1{\relax#1^\ast}
\edef\-#1{\relax\noexpand\ifmmode {\noexpand\bar{#1}} \noexpand\else \-#1\noexpand\fi}
\def\>#1{\vec{#1}}
\def\.#1{\dot{#1}}

\def\atop{\@@atop}
\def\%#1{\mathcal{#1}}

\renewcommand{\leq}{\leqslant}
\renewcommand{\geq}{\geqslant}
\renewcommand{\phi}{\varphi}

\newcommand{\I}{\mathop{{}\mathrm{I}}\mathopen{}}
\newcommand{\dtv}{\mathop{d_{\mathrm{TV}}}\mathopen{}}

\newcommand\indep{\protect\mathpalette{\protect\@indep}{\perp}}
\def\@indep#1#2{\mathrel{\rlap{$#1#2$}\mkern2mu{#1#2}}}

\def\parsetime#1#2#3#4#5#6{#1#2:#3#4}
\def\parsedate#1:20#2#3#4#5#6#7#8+#9\empty{20#2#3-#4#5-#6#7 \parsetime #8}
\def\moddate{\expandafter\parsedate\pdffilemoddate{\jobname.tex}\empty}

\makeatother

\begin{document}

\title [Stein's method and Narayana numbers] {Stein's method and Narayana numbers}

\author{Jason Fulman}

\address{Department of Mathematics\\
        University of Southern California\\
        Los Angeles, CA, 90089, USA}
\email{fulman@usc.edu}

\author{Adrian R\"ollin}
\address{Department of Statistics and Applied Probability\\
National University of Singapore\\
        Singapore}
\email{adrian.roellin@nus.edu.sg}



\date{Version of May 2020}

\begin{abstract} Narayana numbers appear in many places in combinatorics and probability, and it is known that they are asymptotically normal. Using Stein's method of exchangeable pairs, we provide an error of approximation in total variation to a symmetric binomial distribution of order~$n^{-1}$, which also implies a Kolmogorov bound of order~$n^{-1/2}$ for the normal approximation. Our exchangeable pair is based on a birth-death chain and has remarkable properties, which allow us to perform some otherwise tricky moment computations. Although our main interest is in Narayana numbers, we show that our main abstract result can also give improved convergence rates for the Poisson-binomial and the hypergeometric distributions.

\end{abstract}

\maketitle

\section{Introduction}

\noindent We use the convention that the Narayana numbers~$N(n,k)$ are defined as
\[ N(n,k) = \frac{1}{n} {n \choose k-1} {n \choose k} \ , \ 1 \leq k \leq n \]
(some authors define them as~$\frac{1}{n} {n \choose k} {n \choose k+1}$, where
$0 \leq k \leq n-1$). The Narayana numbers refine the Catalan numbers~$C_n =
\frac{1}{n+1} {2n \choose n}$, since  \[ \sum_{k=1}^n N(n,k) = C_n. \]
The Catalan numbers are ubiquitous; see the book by \cite{Stan}
for~214 objects enumerated by Catalan numbers. The Narayana numbers also
appear in interesting places, and a good discussion of them is given by \cite[Chapter~2]{Pet}. Some places (there are many others!) in combinatorics
and probability where the Narayana numbers appear are: enumerating Dyck paths by peaks,
enumerating 231-avoiding permutations by descents and enumerating non-crossing
set partitions by the number of blocks (see \cite{Pet}), in the
stationary distribution for the partially symmetric exclusion process
(see \cite{Cor}), and in the enumeration of totally positive Grassmann cells (see \cite{Wil}).

Given these appearances of the Narayana numbers, it is natural to study their
limiting distribution. We define a probability distribution
$\pi$ on the set~$\{1,\dots,n\}$ by
\[ \pi(k) = \frac{N(n,k)}{C_n}  = \frac{(n+1) {n \choose k} {n \choose k-1}}{n {2n \choose n}}. \]
We let~$K$ be a random variable which is equal to~$k$ with probability~$\pi(k)$,
and we define a random variable~$W$ by
\[ W = \frac{K - \mu_n}{\sigma_n}, \] where
\[ 
  \mu_n = \frac{n+1}{2}, 
  \qquad 
  \sigma_n^2 = \frac{(n-1)(n+1)}{4(2n-1)}.
\]
It is observed by \cite{Chen2017} that~$W$ has mean~$0$ and variance~$1$, and is
asymptotically normal. This follows from Harper's method for proving
central limit theorems for numbers with real rooted generating functions
by representing them as a sum of independent random Bernoulli random variables (see \cite{Pit} for a thorough survey of the method), together with the fact that the generating function for Narayana
polynomials has real roots; see \cite[Section 4.6]{Pet}. This approach combined with, for example, the classical Berry-Esseen bound or even a refined translated Poisson approximation of independent indicators, will  yield a rate of convergence of~$\sigma_n^{-1}$, which is order~$n^{-1/2}$; see, for example, \cite[Example~3.3]{Rollin2007a}. However, using a direct exchangeable pairs approach, we show in this article that one can obtain much better rates.

First, we define the \emph{translated (almost) symmetric binomial distribution}, which will serve as the approximating distribution. While, ideally, we would want to approximate~$K$ by the symmetric  binomial distribution~$\Bi(n,1/2)$ with~$n$ chosen to match the variance of~$K$ and shifted appropriately to match the mean of~$K$, the restriction that both~$n$ and the shift have to be integer-valued requires some care in the exact definition.

For any real number~$x$, let~$\ceil{x}$ be the smallest integer that is larger or equal to~$x$, let~$\floor{x}$ be the largest integer that is smaller than or equal to~$x$, and let~$\angle{x} = x-\floor{x}$. Note that~$x = \floor{x}+\angle{x}$ and~$x = \ceil{x}-\angle{-x}$. Assume~$\mu\in \IR$ and~$\sigma^2>0$ are given. Let~$\delta = \angle{-4\sigma^2}$, so that~$\ceil{4\sigma^2}=4\sigma^2+\delta$. Moreover, let~$t=\angle{-\mu+2\sigma^2+\delta/2}/\ceil{4\sigma^2}$. Denote by~$\cBi(\mu,\sigma^2)$ the binomial distribution~$\Bi(\ceil{4\sigma^2},1/2-t)$ shifted by~$\mu-\ceil{4\sigma^2}(1/2-t)$.

It is not difficult to check that if~$X\sim\cBi(\mu,\sigma^2)$, then~$X$ is integer-valued, that~$\IE X = \mu$ and that~$\sigma^2-1/(4\sigma^2)\leq\Var X\leq\sigma^2+1/4$; in the context of distributional approximation, we like to think of~$\cBi(\mu,\sigma^2)$ as a discrete analogue to the normal distribution with mean~$\mu$ and variance (almost)~$\sigma^2$. The fact that we cannot match the variance exactly introduces only a very small error in the setting we are concerned with.

Finally, for probability measures~$P$ and~$Q$ on~$\IZ$, define the total variation metric
\be{
  \dtv(P,Q) = \sup_{A\subset\IZ}\abs{P[A]-Q[A]}
}
One purpose of this article is to prove the following explicit result, where~$\law(K)$
denotes the law of~$K$. 

\begin{theorem} \label{thm1}For~$K$ defined as above, we have
\be{
  \dtv\bclr{\law(K),\cBi(\mu,\sigma^2)}  \leq \frac{12}{n}. %, \\
%  \dloc\bclr{\law(K),\TP(\mu,\sigma^2)} & \leq ???,
}
\end{theorem}

It is possible to deduce a Berry-Esseen-type bound from Theorem~\ref{thm1}, but the rates of convergence for an integer-valued random variable to the normal distribution in Kolmogorov distance can never be better than the scaling factor.

\begin{cor} \label{cor1} There is a universal constant~$C$ such that
\ben{\label{1}
\sup_{x\in\IR} \abs{ \IP[W \leq x] - \Phi(x)}\leq \frac{C}{n^{1/2}},
}
where~$\Phi(x)$ is the standard normal distribution.
\end{cor}

\begin{remark}
(1) The bound~\eqref{1} also follows from the classical Berry-Esseen bound for sums~$S_n =\sigma^{-1}( X_1+\dots+X_n)$ of centered and independent, but non-identically distributed random variables; we have
\ben{\label{2}
  \sup_{x\in\IR} \abs{ \IP[S_n \leq x] - \Phi(x)}\leq
  \frac{C_0\sum_{i=1}^n\IE\abs{X_i}^3}{\sigma^{3}},
}
where~$C_0$ is no bigger than~$0.5606$ as was shown by \cite{Shevtsova2010}. Since the distribution of~$K$ can be represented as a sum of independent indicator random variables, and since the third central moment of an indicator random variable can always be upper-bounded by its variance, we obtain from~\eqref{2} that the left hand side of~\eqref{1} is bounded by~$C_0\sigma_n^{-1}$, which yields that, on the right hand side of~\eqref{1}, we can take~$C=1.59$, since~$\sigma_n\geq n^{1/2}/(2\sqrt{2})$ for~$n\geq 2$, as is easy to prove.

(2) Applying a result of \cite{SS} to the exchangeable pair in this paper, one can show that
one can take~$C=10$ in Corollary~\ref{cor1}. The calculations involved are very similar to those in the present
paper, but we omit the details as the bound in the previous remark is sharper.

\end{remark}

We can also deduce a local limit theorem from Theorem~\ref{thm1}, since the difference of the point probabilities are upper bounded by the total variation distance and since the total variation rates are better than~$\sigma_n^{-1}$. It should also be fairly easy to prove a local limit theorem directly using Stirling's approximation.

\begin{cor}
There is a universal constant~$C$ such that
\be{
  \sigma_n^{1/2}\,\sup_{k\in \IZ}\bbabs{\IP[K=k]-\frac{1}{\sigma_n}\phi\bbclr{\frac{k-\mu_n}{\sigma_n}}}
  \leq \frac{C}{n^{1/2}},
}
where~$\phi(x)$ is the standard normal density.
\end{cor}

We will apply the following result, which we will prove later using similar ideas as those of \cite{Rollin2007a,Rollin2008a} and \cite{Barbour2018}. For its statement (and for the rest of the
paper), recall that a pair of random variables $(X,X')$ on a state space is called exchangeable if for all~$x_1$ and~$x_2$, we have $\IP[X=x_1,X'=x_2]=\IP[X=x_2,X'=x_1]$. Also, as is typical in probability theory, let~$\IE(X|Y)$ denote the expected value of~$X$ given~$Y$.

\begin{theorem}\label{thm2} Assume~$(X,X')$ is an exchangeable pair of integer-valued random variables with~$\IE X = \mu$ and~$\Var X=\sigma^2$, such that~$X'-X\in\{-1,0,1\}$ almost surely and such that 
\ben{\label{3}
  \IE\cls{X'-\mu\given X} = (1-\lambda)(X-\mu).
}
Then, with~$S = S(X)= \IE\clc{\I[X'\neq X]\given X}$,
\ben{
  \dtv\bclr{\law(X),\cBi(\mu,\sigma^2)}
  \leq \frac{\sqrt{\Var S}}{2\lambda\sigma^2} + \frac{1.4}{\sigma^2}.
  \label{4} %\\
}
Here~$\law(X)$ denotes the law of~$X$.
\end{theorem}

One of our contributions is to show how to apply Theorem~\ref{thm2} to
prove Theorem~\ref{thm1}. This is not straightforward for two reasons.
First, it is not at all obvious how to construct an exchangeable pair
$(K,K')$ satisfying the linearity condition~\eqref{3}. As we
show in Section~\ref{sec1}, we do this using a birth-death chain on the
state space~$\{1,\dots,n\}$. We discovered this birth-death chain through
experimentation. Second, in order to compute~$\Var(S)$, it turns out to be
necessary to know the first four moments of the random variable~$K$.
The generating function for Narayana numbers is complicated.
What is known (see for instance \cite[p.~25]{Pet}) is that the bivariate generating function
\[
 C(t,z) = \sum_{n \geq 0}\sum_{k=1}^n N(n,k) t^{k-1} z^n 
\] 
is equal to
\begin{equation} \label{5} 
  C(t,z) = \frac{1+z(t-1)-\sqrt{1-2z(t+1)+z^2(t-1)^2}}{2tz}.
\end{equation} It is not at all clear how to extract the fourth moment
of~$K$ from~\eqref{5}. We show how to use properties of the
exchangeable pair~$(W,W')$ to compute the first four moments of~$W$
(of course the first and third moments are zero by symmetry).

\subsection{Further applications of Theorem~\ref{thm2}}

To close this section, we give two further applications of Theorem \ref{thm2}, our main abstract result. These applications are to sums of independent random indicators, and to the hypergeometric distribution.

\subsubsection{Poisson-binomial distribution}

We first consider sums of independent random indicators. This case has been considered, for example, by \cite{Rollin2007a} in the context of translated Poisson approximation, which has found applications in algorithmic game theory, see for example the work of \citet{Daskalakis2007} and \citet{Daskalakis2015}. If the probabilities of the indicators are near $1/2$, the resulting sum is close to symmetric which should result in better rates of approximation. The following result confirms this.

\begin{theorem}\label{thm3} Let $X=\sum_{i=1}^n \xi_i$, where $\xi_i\sim \Be(p_i)$ are independent Bernoulli random variables. Then
\ben{\label{6}
  \dtv \bclr{\law(X),\cBi(\mu,\sigma^2)} \leq \frac{1.4+0.5\sqrt{\strut\sum_{i=1}^n (1-2p_i)^2p_i(1-p_i)}}{\sum_{i=1}^n p_i(1-p_i)\strut}
}
\end{theorem}

Note that \cite{Rollin2007a} obtained
\ben{\label{7}
  \dtv \bclr{\law(X),\TP(\mu,\sigma^2)} \leq \frac{2+\sqrt{\strut\sum_{i=1}^n p_i^3(1-p_i)}}{\sum_{i=1}^n p_i(1-p_i)},
}
where $\TP(\mu,\sigma^2)$ denotes a Poisson distribution shifted by an integer amount in such a way that the mean of $X$ is matched exactly and such that the variance is matched subject to rounding constraints; see also \cite{Barbour2018}. The bound \eqref{7} performs better than \eqref{6} if all of the $p_i$ are small, which is clear since that is the regime of Poisson approximation. However, \eqref{6} performs better when many of the $p_i$ are close to $1/2$ due to the factor $(1-2p_i)^2$, in which case the rate of convergence could indeed be better than $1/\sigma$. For example, if $p_i \in 1/2\pm n^{-1/2}$ for all $1\leq i\leq n$, then \eqref{6} yields a rate of order $\bigo(1/n)$ whereas the rate of \eqref{7} is only of order $\bigo(1/\sqrt{n})$. If $p_i=1/2$ for all $i$, the expression under the square root in \eqref{6} vanishes, and the constant $1.4$ could also be removed in this special case, so that the right hand side of~\eqref{6} vanishes, as it ought to.

\begin{proof}[Proof of Theorem~\ref{thm3}] Let $I$ be uniformly distributed on $\{1,\dots,n\}$, let $\xi_1',\dots,\xi_n'$ be independent copies of the $\xi_i$, and let $X' = X - \xi_I + \xi_I'$. Then it is easy to verify that \eqref{3} holds with $\lambda=1/n$. Now, letting $S(\xi)=\IE\clc{\I[X'\neq X]\given \xi}$, we have
\be{
  S(\xi) = \frac{1}{n}\sum_{i=1}^n \bclr{\xi_i(1-p_i)+(1-\xi_i)p_i} = \frac{\mu}{n}+\frac{1}{n}\sum_{i=1}^n (1-2p_i)\xi_i,
}
and \eqref{6} easily follows from Theorem~\ref{thm2}.
\end{proof}

\subsubsection{Hypergeometric distribution}

Denote by $\Hyp(N,n,m)$ the hypergeometric distribution, which is the number of good items when $n$ items are drawn without replacement from $N$ items, out of which $m$ are labelled `good'. Normal approximation of the hypergeometric is a classic; see for example the discussion of \cite{Lahiri2007}. Local limit theorems can be obtained via Stirling's approximation, since the probabilities have an explicit representation in terms of combinatorial factors. Using Harper's method, \cite{Vatutin1983} showed that the hypergeometric distribution can be represented as a sum of independent random indicators, but little is known about the individual probabilities appearing in the sum. Total variation approximations by simpler distributions have been given, for instance, by \cite{Rollin2007a} with rates of order $1/\sigma$. \cite{Mattner2018} proved Berry--Esseen-type bounds with optimal constants. The hypergeometric distribution is symmetric if either $2m=N$ or $2n=N$ (excluding the degenerate cases $m=0$ etc.), and one should expect a better rate of convergence to a symmetric binomial distribution. This is quantified by the next theorem, and to the best of our knowledge, this improved bound is new. 

\begin{theorem}\label{thm4} Let $X\sim\Hyp(N,n,m)$ with $N\geq 4$, $1\leq m < N$ and $1\leq n < N$; then 
\ben{\label{8}
  \dtv \bclr{\law(X),\cBi(\mu,\sigma^2)} \leq \frac{\strut\sqrt{6(N-2 m)^2(N-2 n)^2/N^3 +130}}{2N^{1/2}\sigma}+\frac{1.4}{\sigma^2},
}
where $\sigma^2=\Var X = \frac{mn(N-m)(N-n)}{(N-1)N^2}$.
\end{theorem} 

So in the case where $m$ and $n$ are of order $N$, we have that $\sigma^2$ is of order $N$. If also either $m$ or $n$ is close to $N/2$, we obtain improved bounds from \eqref{8}. Indeed, if $m=N/2$, say, the larger factor under the square root disappears, so that the overall bounds is of order $1/N$, much better than the typical order of $1/N^{1/2}$ one would obtain from \cite{Rollin2007a} or \cite{Mattner2018}.

\begin{proof}[Proof of Theorem~\ref{thm4}] For the exchangeable pair $(X,X')$ satisfying \eqref{3} with $\lambda=\frac{N}{m(N-m+1)}$, we refer to \cite{Rollin2007a}. It suffices to note that
\ba{
  \IP[X'=X+1|X] & = \frac{m-X}{m}\times\frac{n-X}{N-m+1},\\
  \IP[X'=X-1|X] & = \frac{X}{m}\times\frac{N-m-n+X}{N-m+1},
}
which yields
\be{
  S(X) =\lambda \frac{m n + (N-2m-2n)X+2 X^2}{N}.
}
Now, tedious calculations yield that
\bes{
  & \Var ((N-2m-2n)X+2 X^2) \\
  & \quad = (N-2m-2n)^2\Var(X) + 4\Var(X^2) + 4(N-2m-2n)\Cov(X,X^2)\\
  & \quad=\frac{\sigma^2}{(N-3) (N-2) (N-1)}\\
%  & \quad=\frac{mn(N-m)(N-n)}{(N-3) (N-2) (N-1)^2N^2}\\
  & \quad\quad\times \bbcls{(N-2 m)^2(N-2 n)^2 (N-2) \\
  &\quad\quad\qquad +2(N-2 m)(N-2 n)  (4 m-3) (n-1) \\
  &\quad\quad\qquad+4(N-2 m) (m-1) (n-1) (2 m+2 n-1)\\
  &\quad\quad\qquad+8(m-1) m (n-1) (2 m-n-1)-(N-2 m)^2(N-2 n) }\\
  & \quad\leq\sigma^2
   \bbcls{\frac{(N-2 m)^2(N-2 n)^2}{(N-3)(N-1)} + 13N+45N+72N}\\
  & \quad\leq \sigma^2 \bbclr{\frac{6(N-2 m)^2(N-2 n)^2}{N^2} +130N}.
}
The first inequality is obtained by bounding ratios by ``constant times $N$'', using in particular that $N\geq 4$ and $1\leq m,n\leq N$; for instance, $2(N-2 m)(N-2 n)  (4 m-3) (n-1)/((N-3) (N-2) (N-1))\leq 13N$. Applying Theorem~\ref{thm2}, the final bound can be obtained after some straightforward simplifications.
\end{proof}

\section{Proof of Theorem~\ref{thm1}} \label{sec1}

\noindent We now prove our main results for Narayana numbers, using notation from the introduction. Throughout we assume that~$n \geq 2$ to avoid division by zero.
In order to study the asymptotic behaviour of~$K$ by Stein's method, we will construct
an exchangeable pair~$(K,K')$. To do this, we first define a birth-death
chain on the set~$\{1,\dots,n\}$ by
\bg{
  p(k,k+1) = \frac{(n-k)(n-k+1)}{n(n-1)}, \\
  p(k,k-1)  = \frac{k(k-1)}{n(n-1)},\\
 p(k,k)  = \frac{2(k-1)(n-k)}{n(n-1)}.
}
It is easy to see that~$\pi(i) p(i,j) = \pi(j) p(j,i)$ for all~$i$ and~$j$. Thus this
birth death chain is reversible with respect to~$\pi$. This allows us to
construct an exchangeable pair~$(K,K')$ as follows: choose~$K \in \{1,\cdots,n\}$
from~$\pi$ and then obtain~$K'$ by taking one step according to the birth-death chain.

The next result shows that the exchangeable pair~$(K,K')$ satisfies the linearity condition~\eqref{3} of Theorem~\ref{thm2}.

\begin{lemma} \label{lem1}
\[ \IE[K'-\mu|K] = \left( 1 - \frac{2}{n-1} \right) (K-\mu) .\]
\end{lemma}

\begin{proof} We have
\bes{
\IE[K'-K|K] & = p(K,K+1) - p(K,K-1) \\
& =    \frac{(n-K)(n-K+1) - K(K-1)}{n(n-1)} \\
& =  \frac{n-2K+1}{n-1}
 =  - \frac{2}{n-1}\bclr{K - {\textstyle\frac{n+1}{2}}}.\qedhere
}
\end{proof}

As a corollary, we obtain the following result, which is also immediate from the symmetry of the distribution~$\pi$, but it is interesting to deduce it using the pair~$(K,K')$.

\begin{cor} \label{thm5} We have
\be{
\IE K = \frac{n+1}{2}.
}
\end{cor}

\begin{proof} From the proof of Lemma~\ref{lem1} and the exchangeability of
$K$ and~$K'$, one has that
\[ 0 = \IE[K'-K] = \IE[\IE[K'-K|K]] = - \frac{2}{n-1} \IE \left[ K - \frac{n+1}{2} \right] .\]
\end{proof}

Next we use the exchangeable pair to calculate~$\IE K^2$.
The value of~$\IE K^2$ was given by \cite{Chen2017}. Our derivation uses the exchangeable pair~$(K,K')$.

\begin{lemma} \label{lem2} We have
\be{
  \IE K^2 = \frac{n^3+2n^2-1}{4n-2}.
}
\end{lemma}

\begin{proof}
Consider the quantity
\[ \IE[(K')^2-K^2|K]. \]
On one hand, its expected value is equal to~$\IE[(K')^2-K^2] = 0$.
On the other hand, the construction of~$(K,K')$ gives that its
expected value is equal to the expected value of
\be{
  p(K,K+1)\bclr{(K+1)^2-K^2}
 + p(K,K-1)  \bclr{(K-1)^2-K^2}.
}
Using the definition of~$p(K,K+1)$ and~$p(K,K-1)$, it follows that
\[ 0 = \IE[(n-K)(n-K+1)(2K+1) - K(K-1)(2K-1)].\] Expanding this gives
that
\[ \IE[K^2 (2 - 4 n) + n (1 + n) + 2 K (n^2-1)] = 0.\] So we can solve
for~$\IE K^2$ in terms of~$\IE K$, which we computed in Corollary~\ref{thm5}.
\end{proof}

\begin{lemma} \label{lem3} We have
\be{
  \IE K^3 = \frac{(n^2+2n-2)(n+1)^2}{8n-4}.
}
\end{lemma}

\begin{proof} By symmetry of the distribution~$\pi$, we have~$\IE (K-\mu)^3 =0$; that is,
\be{
  \IE \left[ K^3 - 3K^2 \left( \frac{n+1}{2} \right) + 3K  \left( \frac{n+1}{2} \right)^2 -
\left( \frac{n+1}{2} \right)^3 \right] = 0.
}
Thus
\[ \IE K^3 = 3 \left( \frac{n+1}{2} \right) \IE K^2  - 3 \left( \frac{n+1}{2} \right)^2 \IE K
+ \left( \frac{n+1}{2} \right)^3 .\] The lemma now follows from Corollary~\ref{thm5} and
Lemma~\ref{lem2}.
\end{proof}

\begin{lemma} \label{lem4} We have
\be{\IE K^4 = \frac{(n^5+4n^4-3n^3-12n^2+2n+6)(n+1)}{4(2n-1)(2n-3)}
}
\end{lemma}

\begin{proof} Consider the quantity
\[ \IE[(K')^4-K^4|K]. \] On one hand, its expected value is equal to~$\IE[(K')^4-K^4]=0$.
On the other hand, the construction of~$(K',K)$ gives that its expected value is
equal to the expected value of
\bes{
& p(K,K+1) \bclr{ (K+1)^4 - K^4}
+ p(K,K-1) \bclr{ (K-1)^4 - K^4}\\
&\quad =
\frac{(n-K)(n-K+1)}{n(n-1)} \bclr{ (K+1)^4 - K^4}
+ \frac{K(K-1)}{n(n-1)} \bclr{ (K-1)^4 - K^4}
}
Since the~$K^5$ and~$K^6$ terms cancel out, this equality lets us solve for~$\IE K^4$ in terms of~$\IE K^3$,~$\IE K^2$ and~$\IE K$ (which
were computed in Lemmas~\ref{lem3} and~\ref{lem2} and Corollary~\ref{thm5}), yielding the final expression.
\end{proof}

Next, we give an exact formula for~$\Var S$, where
\[ S=S(K) = \IE[(K'-K)^2|K]. \]

\begin{lemma} \label{lem5} For~$n \geq 2$,
\be{
  \Var S = \frac{(n+1)(n-2)  }{(2n-1)^2(2n-3)(n-1)}.
}
\end{lemma}

\begin{proof} Clearly
\bes{
S(K) & =  p(K,K+1) + p(K,K-1) \\
& = 1 - p(K,K)
 =   1 - \frac{2(K-1)(n-K)}{n(n-1)} .
}
Thus
\begin{equation} \label{9}
\Var S = \frac{4}{ n^2 (n-1)^2} \Var[(K-1)(n-K)].
\end{equation}
To compute the variance of~$(K-1)(n-K)$, we compute~$\IE[(K-1)(n-K)]$
and~$\IE[(K-1)^2(n-K)^2]$.

It follows from Corollary~\ref{thm5} and Lemma~\ref{lem2} that
\be{
\IE[(K-1)(n-K)] = - \IE K^2  + (n+1) \IE K - n = \frac{n(n-1)(n-2)}{(4n-2)}.
}
Similarly, since
$\IE[(K-1)^2(n-K)^2]$ is equal to
\be{
  \IE K^4  - (2n+2) \IE K^3  + (n^2+4n+1) \IE K^2 - (2n^2+2n) \IE K  + n^2.
}
we can apply Corollary~\ref{thm5} and Lemmas~\ref{lem2},~\ref{lem3}, and
~\ref{lem4}, to conclude that
\[
\IE\bcls{(K-1)^2(n-K)^2} = \frac{n^2(n^4-7n^3+19n^2 - 23n + 10)}{4(4n^2-8n+3)}.
\]

Thus
\[ \Var[(K-1)(n-K)] = \frac{(n+1)n^2(n-1)(n-2)  }{4(2n-1)^2(2n-3)} .\]
The lemma now follows from~\eqref{9}.
\end{proof}

Putting the pieces together, we now prove our main result on Narayana numbers.

\begin{proof}[Proof of of Theorem~\ref{thm1}] The theorem clearly holds for~$n \leq 2$. For~$n > 2$, we apply Theorem~\ref{thm2}. By Lemma~\ref{lem1},~$\IE[K'-\mu|K]=(1- \lambda) (K-\mu)$ with~$\lambda = \frac{2}{n-1}$. It follows easily from Lemma~\ref{lem5} that
\be{
  \dtv\bclr{\law(K),\cBi(\mu,\sigma^2)}
  \leq \frac{1}{n}\sqrt{\frac{n^2(n-2)  }{(2n-3)(n-1)(n+1)}} + \frac{5.6(2n-1)}{(n-1)(n+1)}
}
It is not difficult to see that the fraction inside the square root is less than~$1/2$ and that~$5.6n(2n-1)(n-1)^{-1}(n+1)^{-1} \leq 11.2$ for~$n\geq 2$.
\end{proof}

\section{Proof of Theorem~\ref{thm2}}

\noindent We only give a very compact proof, since much of the material is explained in detail by \cite{Rollin2007a,Rollin2008a}.

\begin{proof}[Proof of Theorem~\ref{thm2}]
Recall the definition of~$\cBi(\mu,\sigma^2)$ with the choices of~$n$,~$\delta$ and~$t$.
\cite[(2.8)]{Rollin2008a} showed that~$Z$ has distribution~$\cBi(\mu,\sigma^2)$ if and only if
\be{
  \IE\bclc{(Z-\mu)\Theta g(Z)-\sigma^2\Delta g(Z)+a(Z)\Delta g(Z)} = 0
}
for all, say, bounded functions~$g:\IZ\to\IR$,
where~$\Theta g(k) = (g(k+1)+g(k))/2$, where~$\Delta g(k) = g(k+1)-g(k)$, and where~$a(k) = nt^2-(k-\mu)t-\delta/4$.
This motivates the definition of the Stein operator
\be{
  (\cB g)(k) := (k-\mu)\Theta g(k)-\sigma^2\Delta g(k)+a(k)\Delta g(k)
}
and setting up the Stein equation
\ben{\label{10}
  (\cB g)(k) = \I[k\in A] - \IP[Z\in A]
}
for~$A\subset \cT:=\{0,\dots,n\}+\mu-n(1/2-t)$ and~$Z\sim\cBi(\mu,\sigma^2)$. 

Now, \cite[(2.4) and (2.8)]{Rollin2008a} showed that there is a solution~$g_A$ to~\eqref{10} that satisfies
\ben{\label{11}
   \norm{\Delta g_A}_\infty \leq 1\wedge\frac{1}{\sigma^2}.
}
Using identity~\eqref{10}, the triangle inequality and~\eqref{11}, we obtain
\besn{\label{12}
  &\dtv\bclr{\law(X),\cBi(\mu,\sigma^2)}
   = \sup_{A\subset \IZ}\abs{\IP[X\in A]-\IP[Z\in A]}\\
  &\quad = \sup_{A\subset \cT}\abs{(\cB g_A)(X)} + \IP[X\not\in\cT]\\
  &\quad\leq \sup_{A\subset \cT}\babs{\IE\clc{(X-\mu)\Theta g(X)-\sigma^2\Delta g(X)}} + \frac{\IE \abs{a(X)}}{\sigma^2} + \IP[X\not\in\cT].
}

We can use Chebychev's inequality to bound
\ben{\label{13}
  \IP[X\not\in\cT]
  \leq\IP[\abs{X-\mu}\geq 2\sigma^2-1]
  \leq \frac{\sigma^2}{(2\sigma^2-1)^2}\leq \frac{0.61}{\sigma^2},
}
where the last inequality holds as long as~$\sigma^2\geq 1.4$, which we may assume without loss of generality since otherwise~\eqref{4} is trivial. Moreover,
\ben{\label{14}
  \IE\abs{a(X)}\leq nt^2+\sigma t+\delta/4 \leq 3/4
}
(see \cite[after (2.17)]{Rollin2008a} for the second inequality).

It remains to bound the first expression on the right hand side of~\eqref{12}; we follow the line of argument of~\cite{Rollin2007a}. Using exchangeability and anti-symmetry,
\besn{\label{15}
  0 & = \IE\bclc{(X'-X)\bclr{g(X')+g(X)}} \\
  &= \IE\bclc{(X'-X)\bclr{g(X')-g(X)}} + 2\IE\clc{(X'-X)g(X)}
}
Using~\eqref{3}, the second term equals
\be{
  2\IE\clc{(X'- X)g(X)} = -2\lambda\IE\clc{(X-\mu)g(X)}.
}
so that~\eqref{15} can be written as
\ben{\label{16}
  \IE\clc{(X-\mu)g(X)} = \frac{1}{2\lambda}\IE\bclc{(X'-X)\bclr{g(X')-g(X)}}.
}
To simplify the right hand side of~\eqref{16}, let~$I_i :=\I[X'-X=i]$ for~$i\in\{-1,+1\}$, and making the case distinction whether~$X'-X=+1$ or~$-1$, write
\bes{
  & \IE\bclc{(X'-X)\bclr{g(X')-g(X)}}
   = \IE\clc{I_{+1}\Delta g(X)} + \IE\clc{I_{-1}\Delta g(X-1)}.
}
Using exchangeability,
\be{
  \IE\clc{I_{-1}\Delta g(X-1)} =
  \IE\clc{I_{+1}\Delta g(X)},
}
so that~\eqref{16} yields
\ben{\label{17}
  \IE\clc{(X-\mu)g(X)} = \frac{1}{\lambda}\IE\clc{I_{+1}\Delta g(X)}
}
Replacing~$g(X)$ by~$g(X+1)$ and using exchangeability again,
\ben{\label{18}
  \IE\clc{(X-\mu)g(X+1)} = \frac{1}{\lambda}\IE\clc{I_{+1}\Delta g(X+1)}
  = \frac{1}{\lambda}\IE\clc{I_{-1}\Delta g(X)}
}
Adding~\eqref{17} and~\eqref{18} and dividing by two,
\be{
  \IE\clc{(X-\mu)\Theta g(X)} = \frac{1}{2\lambda}\IE\clc{(I_{+1}+I_{-1})\Delta g(X)} = \frac{1}{2\lambda}\IE\clc{S\Delta g(X)},
}
and it follows that
\ben{\label{19}
  \IE\clc{(X-\mu)\Theta g(X)-\sigma^2\Delta g(X)} = \IE\bbclc{\bbclr{\frac{S}{2\lambda}-\sigma^2}\Delta g(X)}.
}

Now, noticing that~$\IE S = 2\lambda\sigma^2$, the right hand side of \eqref{19} can be bounded by
\ben{\label{20}
  \frac{\sqrt{\Var S}}{2\lambda}\norm{\Delta g}_{\infty} \leq \frac{\sqrt{\Var S}}{2\lambda\sigma^2}.
}
Combining bounds~\eqref{13},~\eqref{14}, and~\eqref{20} with~\eqref{12} yields the claim.
\end{proof}

\section*{Acknowledgements} 

\noindent JF was partially supported by Simons Foundation Grant 400528, and AR was partially supported by NUS Research Grant R-155-000-167-112. The authors thank Kyle Petersen for helpful discussions.

\end{document}